\documentclass[14pt]{amsart}
\usepackage{amscd,amsthm,amsmath,amssymb}
\usepackage[dvipdf]{graphicx}
\usepackage[matrix,arrow]{xy}
\usepackage{svg}

\makeatletter\@addtoreset{equation}{section}\makeatother

\newtheorem{theorem}[equation]{Theorem}
\newtheorem{prop}[equation]{Proposition}
\newtheorem{lemma}[equation]{Lemma}

\newtheorem{theorem-definition}[equation]{Theorem-definition}

\theoremstyle{definition}
\newtheorem{example}[equation]{Example}

\newtheorem{definition}[equation]{Definition}

\theoremstyle{remark}
\newtheorem{remark}[equation]{Remark}

\makeatletter\@addtoreset{subsection}{equation}\makeatother

\newcommand{\p}{\mathbb{P}}
\newcommand{\cel}{\mathbb{Z}}

\newcommand{\com}{\mathbb{C}}

\newcommand{\ra}{\mathbb{Q}}

\begin{document}

\title{Surjective rational maps and del Pezzo surfaces}

\author{Ilya Karzhemanov and Anna Lekontseva}

\address{\newline{\normalsize Laboratory of AGHA, Moscow Institute of Physics and Technology, 9 Institutskiy per., Dolgoprudny,
Moscow Region, 141701, Russia}\medskip
\newline{\it E-mail address}: karzhemanov.iv@mipt.ru\medskip
\newline{\it E-mail address}: lekontseva.am@phystech.edu}

\begin{abstract}
We study surjective (not necessarily regular) rational
endomorphisms $f$ of smooth del Pezzo surfaces $X$. We prove that
under certain natural non\,-\,degeneracy condition $f$ can have
degree bigger than $1$ only when $(-K_X^2) > 5$. Some structural
properties of $f$ in the case $X = \p^2$ are also established.
\end{abstract}

\medskip

\thanks{{\it MS 2010 classification}: 14E05, 14J26, 14E30}


\sloppy

\maketitle

\bigskip

\section{Introduction}
\label{section:int}

\refstepcounter{equation}
\subsection{}
\label{subsection:int-1}

In the paper \cite{kz}, the study of \emph{surjective rational
self\,-\,maps} $f$ of projective spaces $\p^n$ has been initiated,
and some algebro\,-\,geometric properties of $f$ were established
(cf. \cite{ilya}, \cite{ilya-quant} and \cite{kul-zhd} for further
developments of the subject). The present note extends these
considerations to the case of \emph{algebraic surfaces}.

Namely, let $X$ be a smooth projective surface over $\com$, and
$f: X \dashrightarrow X$ a surjective (not necessarily regular)
rational endomorphism of $X$. In other words, we assume that there
is a \emph{surjective morphism} $f: X \setminus Z \longrightarrow
X$, where $Z \subset X$ is a finite collection of points. Let also
$\deg f$ be the topological degree of $f$ (i.\,e. the number of
points in a general fiber of $f$). Then the basic question one may
address is when does such pair $(X,f)$, with $\deg f > 1$, exist?

We provide an answer to the latter question in the case of del
Pezzo surfaces:

\begin{theorem}
\label{theorem:main} Let $X$ and $f$ be as above. Assume also that
the anticanonical divisor $-K_X$ is ample and $f$ is
non\,-\,degenerate (see Definition~\ref{definition:def-non-deg}
below). Then $\exists f \ : \ \deg f
> 1 \Longleftrightarrow (-K_X^2)
> 5$.
\end{theorem}

\begin{remark}
\label{remark:disc-toric} Recall that $(-K_X^2) > 5$ iff the del
Pezzo surface $X$ is \emph{toric}. Hence it admits plenty of
(Frobenius) regular endomorphisms in this case. At the same time,
when $X = \p^2$ there exists a \emph{non\,-\,regular} surjective
non\,-\,degenerate $f$ (see Example~\ref{example:surj-f-p-2}
below), and it would be interesting to find similar maps for other
toric del Pezzos (cf. Remark~\ref{remark:f-star-of-nef}). Finally,
since all surfaces we consider are rational, they carry a lot of
\emph{non\,-\,surjective} dominant rational maps.
\end{remark}

\refstepcounter{equation}
\subsection{}
\label{subsection:int-2}

Let us recall that Theorem~\ref{theorem:main} was proved in
\cite{beauville} under the assumption that $f$ is \emph{regular}
(all such maps are automatically non\,-\,degenerate due to the
projection formula stated further). However, the arguments in
\emph{loc.\,cit.} do not extend immediately to our,
non\,-\,regular, situation (cf. Remark~\ref{remark:disc-toric}),
since the proof of \cite[Proposition 4]{beauville} relies heavily
on the ramification $K_X = f^*(K_X) + R$ and projection $f_*f^* =
(\deg f)\, \text{Id}$ formulae for $f^*$ and $f_*$ acting on the
N\'eron\,--\,Severi group $N^1(X)$.\footnote{~We refer to the book
\cite{lazar} for all relevant notions and facts of
``\,numerical\,'' algebraic geometry used in our paper.} Although
one might expect an analog of the identity $K_X = f^*(K_X) + R$ in
our case, as shown in Section~\ref{section:add} below for $X =
\p^2$, there is no hope to define $f_*$ properly. Yet, there
\emph{is} an avatar of $f^*$ for any surjective rational map (see
{\ref{subsection:pr-1}}), and so we can take the route as in the
paper \cite{totaro} towards the proof of
Theorem~\ref{theorem:main}.

Namely, after replacing $f$ by a suitable iterate $f \circ \ldots
\circ f$, one may assume that $f^*$ preserves the extremal rays
$R_i$ of the (\emph{rational polyhedral}) Mori cone of $X$ (cf.
our Proposition~\ref{theorem:f-star-def} and the proof of Lemma
6.2 in {\it op.\,cit.}; here we need the extra non\,-\,degeneracy
assumption in Theorem~\ref{theorem:main}). Then the considerations
can be reduced to the case $(-K_X^2) = 5$ (again as in \cite[Lemma
6.2]{totaro} --- see {\ref{subsection:pr-2}} below). We have $R_i
= \mathbb{R}_{\ge 0}[E_i]$ for the $(-1)$\,-\,curves $E_i \subset
X$. These $E_i$ are actually $10$ lines on $X$ with respect to the
anticanonical embedding in $\p^5$ (see e.\,g.
\cite[8.5.1]{dolgachev} for the Petersen graph configuration of
$E_i$). Also there are $5$ pencils of conics on $X$. They define
$5$ morphisms $\pi_j: X \longrightarrow \p^1$ with connected
fibers corresponding to certain faces of the Mori
cone.\footnote{~Recall that there is a birational contraction
$\sigma: X \longrightarrow \p^2$ of four disjoint $E_j$. Let them
be $E_1,\ldots,E_4$, say. Then the classes of conics on $X$ are
$\sigma^*(\ell) - E_j$ and $\sigma^*(2\ell) - E_1 - \ldots - E_4$
for a line $\ell \subset \p^2$.} Then, since $f^*(R_i) = R_i$ by
assumption, we get $\pi_j \circ f = h_j \circ \pi_j$ for some
$h_j: \p^1 \longrightarrow \p^1$ (cf. {\ref{subsection:pr-3}}).
Moreover, following the proof of \cite[Lemma 6.3]{totaro}, we show
that $h_j^{-1}(\Delta_j) \subseteq \Delta_j$ for the set of
critical values of $\pi_j$ (see Lemma~\ref{theorem:sing-fibs}
below). Finally, using the equality $-2K_X = \displaystyle \sum_{i
= 1}^{10} E_i$, we obtain that $f$ induces an
\emph{int\,-\,amplified} endomorphism on the base $\p^1 =
\pi_{j_0}(X)$ for some $j_0$ (see
Lemma~\ref{theorem:sing-fibs-111}). The latter can be seen to be
impossible exactly as in the proof of \cite[Proposition
6.4]{totaro}.

\bigskip

\thanks{{\bf Acknowledgments.}
We are grateful to Ilya Zhdanovskiy for his interest and fruitful
conversations. The work of the first author was carried out at the
Center for Pure Mathematics (MIPT) and was partially supported by
the Russian Science Foundation under grant \textnumero\,
25-21-00083 (https://rscf.ru/project/25-21-00083/).

\bigskip

\section{Proof of Theorem~\ref{theorem:main}}
\label{section:pr}

\refstepcounter{equation}
\subsection{}
\label{subsection:pr-1}

Let $Z \subset X$ be a finite subset such that the restriction
$f\big\vert_{X \setminus Z}$ is regular and $f(X \setminus Z) =
X$. For any curve $C \subset X$, we put $f^*(C) :=
\overline{f^{-1}(C)}$, the Zariski closure of the preimage
$f^{-1}(C)$ (note that the latter is also a curve (maybe affine)
on $X$). Further, if two curves $C_1,C_2 \subset X$ are linearly
equivalent, that is $C_1 = C_2$ as elements of the Picard group
$\text{Pic}(X)$, then we have the equality of divisors
$$
C_1 - C_2 = (\phi)
$$
for some rational function $\phi \in \com(X)$. It is immediately
checked (in local equations) that $f^*(C_1) - f^*(C_2) =
(f^*(\phi))$ and so $f^*$ extends to a \emph{linear}
transformation of $\text{Pic}(X) \otimes_{\cel} \mathbb{R} =
N^1(X)$.

\begin{definition}
\label{definition:def-non-deg} Let us call $f$
\emph{non\,-\,degenerate} if $f^*$ induces an automorphism of
$N^1(X)$ and for any ample class $A \in N^1(X)$, if $A = f^*(B)$
for some big $B$, then $B \in \text{Nef}(X)$.
\end{definition}

\begin{remark}
\label{remark:f-star-of-nef} As we have already mentioned in
Introduction, every regular $f$ is non\,-\,degenerate, likewise
every surjective rational endomorphism of $\p^2$. At the same
time, contrary to our initial claim, there do exist surjective
self\,-\,maps of del Pezzo surfaces, which fail the
non\,-\,degeneracy property (cf. Remark~\ref{remark:disc-toric}).
Namely, take $X := \p^1 \times \p^1$ and (surjective) morphism
$\varphi: \p^2 \setminus \Sigma \longrightarrow X$ given by
$$\varphi([x:y:z]) := [xy:x^2-y^2] \times [xz:x^2-z^2],$$ where
$\Sigma := \{[0:1:0],[0:0:1]\}$. Let also $\psi: X \longrightarrow
\p^2$ be the projection from a point not contained in $X$. Then $f
:= \varphi \circ \psi : X \dashrightarrow X$ is clearly surjective
with $Z = \psi^{-1}(\Sigma)$. The preceding considerations yield a
factorization $$f^*: N^1(X)  \stackrel{\psi^*}{\to} N^1(\p^2)
\simeq \mathbb{R} \stackrel{\varphi^*}{\to} N^1(X)$$ into linear
maps. This shows that $f$ \emph{can not} be non\,-\,degenerate for
$N^1(X) \simeq \mathbb{R}^2$. A more elaborate example (found by
AI), showing that Theorem~\ref{theorem:main} fails for degenerate
maps, is as follows. Take four points $p_1,\ldots,p_4 = [1:\pm 1:
\pm 1]$ on $\p^2$ and consider the map $[x:y:z] \mapsto
[x^3:y^3:z^3]$. We have $h(p_i) = p_i$ for all $i$ and $h$ does
not ramify at $p_i$. Let $X \longrightarrow \p^2$ be the
blow\,-\,up of $p_1,\ldots,p_4$ with exceptional curves
$E_1,\ldots,E_4$. By construction $X$ is del Pezzo of degree $5$.
Moreover, since $h$ preserves the ideal of each point $p_i$, it
lifts (via simple local argument) to a \emph{surjective} map $f: X
\dashrightarrow X$ such that $f(E_i) = E_i$ for all $i$.
\end{remark}

We will assume in what follows that $f$ is non\,-\,degenerate and
$(-K_X^2) \le 5$. Then, letting $\text{Nef}(X)$ be the nef cone of
$X$, we obtain

\begin{prop}
\label{theorem:f-star-def} The linear map $f^*: N^1(X) \to N^1(X)$
satisfies $f^*(\mathrm{Nef}(X)) = \mathrm{Nef}(X)$.
\end{prop}

\begin{proof}
Let us start with the next

\begin{lemma}
\label{theorem:f-star-def-lemma-1} $f^*(\mathrm{Nef}(X)) \subseteq
\mathrm{Nef}(X)$.
\end{lemma}

\begin{proof}
Consider some nef $\ra$\,-\,divisor $H \in \mathrm{Nef}(X)$. We
can write $H = \displaystyle\lim_{k \to \infty} \alpha_k$ for
ample divisors $\alpha_k$ (e.\,g. $\alpha_k = H - K_X/k$) via
Kleiman's criterion. Then, since $f^*$ is continuous and the limit
of nef divisors is also nef, it suffices to assume that $H$ is
ample. Furthermore, replacing $H$ by $mH$ with $m \gg 1$, we
reduce to the case of \emph{very ample} $H$.

Let
\begin{equation}
\label{diag} \xymatrix{
& Y \ar[dl]_g \ar[dr]^{h} \\
X \ar@{-->}[rr]^{f} && X}
\end{equation}
be a resolution of indeterminacies of $f$. Here $g$ is a
composition of blow\,-\,ups and $h$ is a regular morphism. Let
$H_1$ and $H_2$ be two generic curves in the linear system $|H|$.
Then the curves $h^*(H_i)$ are \emph{irreducible}. Indeed, if $h =
h_1 \circ h_2$ is the Stein factorization for $h$, with finite
$h_1: Y' \longrightarrow X$ and birational $h_2: Y \longrightarrow
Y'$ (the surface $Y'$ is normal here), then the divisor $h_1^*(H)$
is ample (again by Kleiman) and so $h_1^*(H_i)$ are irreducible.
In turn, $h_2^*h_1^*(H_i) = h^*(H_i)$ are irreducible, since the
linear system $|H|$ is basepoint\,-\,free by assumption and $h_2$
is an isomorphism over a general point on $Y'$.

Now, we have $g_*h^*(H) = f^*(H)$ by construction and both
(distinct) curves $g_*h^*(H_1) \sim g_*h^*(H_2)$ ($\sim f^*(H)$)
are irreducible, which implies that $f^*(H) \in \mathrm{Nef}(X)$.
It remains to observe that rational classes are dense in the cone
$\mathrm{Nef}(X)$.
\end{proof}

Suppose that $f^*(\mathrm{Nef}(X)) \ne \mathrm{Nef}(X)$. Then
there is an \emph{ample} class $\alpha \in \mathrm{Nef}(X)
\setminus f^*(\mathrm{Nef}(X))$ (cf.
Lemma~\ref{theorem:f-star-def-lemma-1}).

\begin{lemma}
\label{theorem:f-star-def-lemma} There exist $\beta, \gamma \in
N^1(X)$ such that $\alpha = f^*(\gamma - \beta)$, and furthermore
$\gamma - \beta \in \mathrm{Nef}(X)$.
\end{lemma}

\begin{proof}
Since $f^*$ is an isomorphism, the cone $f^*(\mathrm{Nef}(X))$ has
non\,-\,empty interior, containing a class $f^*(\beta) \ne 0$.
This gives $t\alpha + (1-t)f^*(\beta) = f^*(\gamma)$ for some big
$\gamma \in N^1(X)$ and $0 \le t \le 1$ (see Figure~\ref{fig-1}).

Write $\gamma = \tilde{\gamma} + e$ for some ample
$\tilde{\gamma}$ and effective $e$. Then, since the Mori cone
$\overline{NE}(X)$ is rational polyhedral (cf.
{\ref{subsection:pr-2}} below), we obtain that $(\tilde{\gamma} -
(1-t_0)\beta) \cdot C > 0$ for some fixed $0 < t_0 \le 1$ and any
curve on $X$. Hence the cycle $\gamma - (1-t_0)\beta$ is big by
Kleiman's criterion. This implies the assertion for $t_0\alpha =
f^*(\gamma - (1-t_0)\beta)$ and non\,-\,degenerate $f$.

\begin{figure}[h]
\includegraphics[scale=1.9]{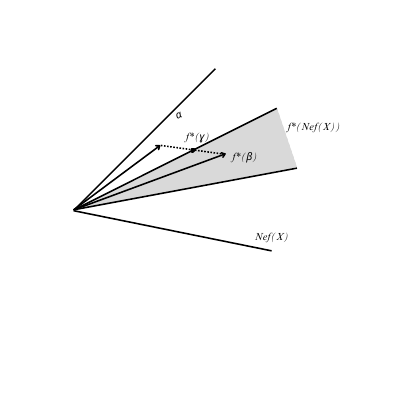}
\caption{~}\label{fig-1}
\end{figure}

\end{proof}

From Lemma~\ref{theorem:f-star-def-lemma} we immediately get
contradiction with the assumption $f^*(\mathrm{Nef}(X)) \ne
\mathrm{Nef}(X)$. This proves
Proposition~\ref{theorem:f-star-def}.
\end{proof}

\refstepcounter{equation}
\subsection{}
\label{subsection:pr-2}

Let $\overline{NE}(X)$ be the Mori cone of $X$ with generating
rays $R_i = \mathbb{R}_{\ge 0}[E_i]$ (cf. {\ref{subsection:int-2}}
and Figure~\ref{fig-2}). Then, since $\overline{NE}(X)$ is the
dual of $\mathrm{Nef}(X)$, it follows from
Proposition~\ref{theorem:f-star-def} that $f^*$ induces a
permutation on the set of $R_i$. Hence, after a finite number of
iterations, one may assume $f^*(R_i) = R_i$ for all $i$.

Fix some $i$ and consider the contraction $\sigma_i: X
\longrightarrow X'$ of $R_i$. Then the property $f^*(R_i) = R_i$
yields a commutative diagram
\begin{equation}
\nonumber \xymatrix{
X\ar@{->}[d]_{\sigma_i}\ar@{-->}[r]^{f}&X\ar@{->}[d]^{\sigma_i}\\
X'\ar@{-->}[r]^{f'}&X'.}
\end{equation}
Here $X'$ is again a del Pezzo surface, with $(-K_{X'}^2) =
(-K_X^2) + 1$, and $f'$ is its rational endomorphism. Moreover,
since $X \setminus E_i \simeq X' \setminus {\{\sigma_i(E_i)\}}$
and $f^{-1}(E_i) \subseteq E_i$ by assumption, the map $f': X'
\dashrightarrow X'$ is surjective; it is also non\,-\,degenerate
because $f^*$ sends hyperplanes to hyperplanes. Thus for the proof
of Theorem~\ref{theorem:main} it suffices to assume that $(-K_X^2)
= 5$.

\refstepcounter{equation}
\subsection{}
\label{subsection:pr-3}

Further, let $C_j \in N^1(X)$, $1 \le j \le 5$, be the classes of
conics on $X \subset \p^5$, satisfying $(C_j^2) = 0$. The linear
systems $|C_j|$ define conic bundles $\pi_j: X \longrightarrow
\p^1$ whose singular fibers are pairs of two lines. It follows
from the formula
$$
e(X) = e(\p^1)e(C_j) + \sum_{p \ : \ \pi_j^{-1}(p) \ \text{is
singular}} e(\pi_j^{-1}(p)) - e(C_j)
$$
for Euler characteristics that each morphism $\pi_j$ has actually
$3$ singular fibers. Hence we have $|\Delta_j| = 3$ for the set
$\Delta_j$ of critical values of $\pi_j$.

Now, for every $j$ we put
$$
C_j^\bot := \{x \in N^1(X) \ : \ C_j \cdot x = 0\},
$$
a hyperplane in $N^1(X)$, which uniquely determines $\pi_j$.
Furthermore, since the divisor $C_j$ is nef, the intersection $F_j
:= C_j^\bot \cap \overline{NE}(X)$ is a \emph{face} of the Mori
cone (see Figure~\ref{fig-2}). In particular, since $f^*(R_i) =
R_i$ for all $i$, we obtain $f^*(F_j) = F_j$ as well. But the
latter precisely means that $f$ preserves every contraction
$\pi_j$. Thus we get $\pi_j \circ f = h_j \circ \pi_j$ for some
$h_j: \p^1 \longrightarrow \p^1$ and all $j$ (cf.
{\ref{subsection:pr-2}}).

\begin{figure}[h]
\includegraphics[scale=1.7]{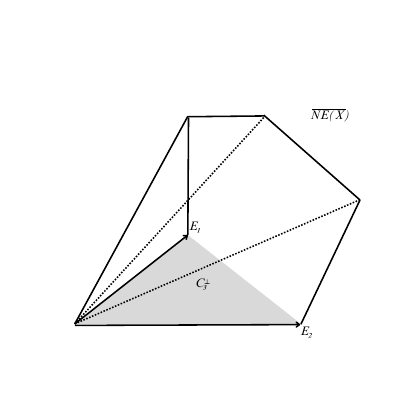}
\caption{~}\label{fig-2}
\end{figure}

The following property follows from \cite[Lemma 6.3]{totaro}:

\begin{lemma}
\label{theorem:sing-fibs} $h_j^{-1}(\Delta_j) \subseteq \Delta_j$.
\end{lemma}

\begin{proof}
This is basically \cite[Lemma 6.3]{totaro}. Namely, if $F$ is a
fiber of $\pi_j$ satisfying $\dim f(F \setminus (F \cap Z)) = 1$
(cf. {\ref{subsection:pr-1}}), then the main point of the argument
in \cite{totaro} is to show that the image of $f\big\vert_F$ is
open and closed in another fiber of $\pi_j$. This holds in our
situation, since $f$ is surjective (hence open) and $F$ is a
curve, so the claim follows.
\end{proof}

We have $f^*(C_j) = m_jC_j$ for some $m_j \in \mathbb{N}$ by
construction. Since $f$ is not ramified along a general fiber $F
\sim C_j$ of $\pi_j$ and $\pi_j \circ f = h_j \circ \pi_j$ as
above, this $m_j$ equals the number of connected components of
$f^{-1}(F)$, which is the same as $\deg h_j$.

\begin{lemma}
\label{theorem:sing-fibs-111} There exists $j_0$ such that
$m_{j_0}
> 1$.
\end{lemma}

\begin{proof}
Suppose the contrary. Thus $f^*(C_j) = C_j$ for all $j$. Then,
since all $C_j$ span $N^1(X)$, we obtain that $f^*(E_i) = E_i$ for
all $i$ as well.

In particular, since $-2K_X = \displaystyle \sum_{i = 1}^{10}
E_i$, we get $f^*(-2K_X) = -2K_X$. But the latter means that $f$
is induced by a projective transformation of the (dual) space
$|-2K_X|^* \supset X$, i.\,e. $\deg f = 1$, a contradiction.
\end{proof}

It follows from Lemma~\ref{theorem:sing-fibs-111} that $h_{j_0}$
is an int\,-\,amplified endomorphism of $\p^1 = \p_{j_0}(X)$ (in
the terminology of \cite{totaro}) and
Lemma~\ref{theorem:sing-fibs}, in turn, yields a \emph{pullback
morphism}
$$
\Omega^1_{\p^1}(\log \Delta_{j_0}) \to
(h_{j_0})_*(\Omega^1_{\p^1}(\log \Delta_{j_0}))
$$
of differential forms with logarithmic poles (cf. \cite[Lemma
2.1]{totaro}). Equivalently, one gets a morphism of line bundles
from $(h_{j_0})^*(K_{\p^1} + \Delta_{j_0})$ to $K_{\p^1} +
\Delta_{j_0}$, hence a global section of $D := K_{\p^1} +
\Delta_{j_0} - (h_{j_0})^*(K_{\p^1} + \Delta_{j_0})$. However,
since $\deg h_{j_0} > 1$, we have $\deg D = 1 - \deg h_{j_0} < 0$
(for $\deg \Delta_{j_0} = 3$), a contradiction.

The proof of Theorem~\ref{theorem:main} is complete.

\bigskip

\section{Addendum}
\label{section:add}

\refstepcounter{equation}
\subsection{}
\label{subsection:add-1}

It is tempting to generalize Theorem~\ref{theorem:main} to the
case of higher\,-\,dimensional Fano varieties $X$. For this,
however, as the exposition of \cite{totaro} reveals, one has to
employ more refined methods (like Bott vanishing for cohomology of
differential forms with logarithmic poles, pullbacks of such forms
via $f^*$, etc.) than our (mostly combinatorial) arguments in
Section~\ref{section:pr}. We do not attempt to address these
matters here and refer to the paper \cite{ilya} for related
results about $X = \p^n$.

For now, let us come back to the discussion in
{\ref{subsection:int-2}} and try to establish a variant of the
ramification formula mentioned there, which is adopted to our
situation. We prove the following:

\begin{prop}
\label{theorem:ram-form} Let $X = \p^2$. Then there is an
\emph{effective} divisor $R$ such that $K_X = f^*(K_X) + R$. In
particular, given a resolution diagram \eqref{diag} and Stein
factorization $h: Y \stackrel{h_2}{\longrightarrow} Y'
\stackrel{h_1}{\longrightarrow} X$ (cf. the proof of
Lemma~\ref{theorem:f-star-def-lemma-1}), we have either $b_j + m_j
\ge 0$ or $Z_j \simeq \p^1$ for all $j$, where
$$
K_Y \equiv h_2^*(K_{Y'}) + \displaystyle \sum_j b_j Z_j
$$
(numerically on $Y$), $Z_j$ are $h_2$\,-\,exceptional curves,
$$
h_2^*(\Delta) = (h_2)_*^{-1}(\Delta) + \displaystyle \sum_j m_jZ_j
$$
for the ramification divisor $\Delta \ge 0$ of $h_1$ and its
birational transform $(h_2)_*^{-1}(\Delta)$.
\end{prop}

\begin{proof}
Recall that in \cite{ilya} an operation $f^*$ on \emph{vector
bundles} was defined. In particular, one has
$f^*(\Omega^1_{\p^2})$, together with an inclusion
$f^*(\Omega^1_{\p^2}) \subseteq \Omega^1_{\p^2}$ (see \cite[Lemma
6]{ilya}). Now, taking the degrees, we arrive at the setting of
{\ref{subsection:pr-1}} above, so that there is a morphism of line
bundles from $f^*(K_{\p^2})$ to $K_{\p^2}$. Hence we get $R :=
K_{\p^2} - f^*(K_{\p^2}) \ge 0$ and the first claim of proposition
follows.

Next we have $K_{Y'} = h_1^*(K_{\p^2}) + \Delta$ by the Hurwitz
formula. There is also a numerical equality $K_Y \equiv
h_2^*(K_{Y'}) + \displaystyle \sum_j b_j Z_j$ for some $b_j \in
\ra$. Thus we obtain
$$
K_Y \equiv h^*(K_{\p^2}) + h_2^*(\Delta) + \sum_j b_j Z_j.
$$
At the same time, we have $g_*(K_Y) = K_{\p^2}$ because $g$ is a
composition of blow\,-\,ups, and $g_*(h^*(K_{\p^2})) =
f^*(K_{\p^2})$ by construction. Then the preceding considerations
imply that
$$
NR = N(K_{\p^2} - f^*(K_{\p^2})) = N(g_*((h_2)_*^{-1}(\Delta)) +
\sum_j (b_j + m_j) g_*(Z_j))
$$
is an effective divisor for sufficiently divisible $N \in \cel$.
In particular, $b_j + m_j \ge 0$ for all $j$, unless $g_*(Z_j) =
0$ (whence $Z_j \simeq \p^1$).
\end{proof}

\begin{example}[{see \cite[Example 1.6]{kz}}]
\label{example:surj-f-p-2} Let us define an endomorphism $f =
[f_0:f_1:f_2]: \p^2 \dashrightarrow \p^2$ with components $f_i$ as
follows: $f_0 := x_0x_2 + x_1^2$, $f_1 := x_1x_2 + x_0^2$, $f_2 :=
x_0^2 + x_1^2$, where $x_i$ are projective coordinates on $\p^2$.
Then $f$ is surjective of degree $3$. Furthermore, the
indeterminacy set of $f$ consists of the single point $[0:0:1]$,
which can be resolved by one blow\,-\,up. Hence we have $Y' = Y =
\mathbb{F}_1$ (for \emph{finite} $h$) in the setting of
Proposition~\ref{theorem:ram-form}. The curve $R$ is an
irreducible cubic with the node at $[0:0:1]$.
\end{example}

\refstepcounter{equation}
\subsection{}
\label{subsection:add-2}

We conclude our paper by posing a couple of problems for future
research:

\begin{itemize}

    \item Characterize non\,-\,regular surjective rational
    endomorphisms of del Pezzo surfaces (compare with \cite[Remark 7]{ilya} and with Remark~\ref{remark:disc-toric} above). To be able do this, one will probably need some
    understanding of singularities of the pair $(Y',\Delta')$ from
    Proposition~\ref{theorem:ram-form},\footnote{~Here one may suggest a relation between the singularities of $(Y',\Delta)$ and rational singularities (cf. \cite[5.1]{KM}).} which is also worth a
    closer study.

    \smallskip

    \item What about surjective rational maps for other surfaces,
    e.\,g. $\text{K3}$ or Abelian ones, whose Mori cone is not rational polyhedral (cf. Remark~\ref{remark:f-star-of-nef} and \cite{karz-kon})?

\end{itemize}

\end{document}